\documentclass[reqno]{amsart}
\usepackage[utf8]{inputenc}
  \usepackage{graphicx}
  \usepackage{amsmath}
  \usepackage{amsthm}
  \usepackage{amssymb}

\usepackage{amsfonts}
\usepackage{fullpage}
\usepackage{hyperref}
\usepackage{enumerate}

\newtheorem{theorem}{Theorem}

\newtheorem{lemma}[theorem]{Lemma}

\newtheorem{definition}[theorem]{Definition}

\title{The area of the Mandelbrot set and Zagier's conjecture}
\author{Patrick F. Bray}
\author{Hieu D. Nguyen}
\date{September 2, 2017}

\begin{document}

\begin{abstract}
       We prove Zagier's conjecture regarding the 2-adic valuation of the coefficients $\{b_m\}$ that appear in Ewing and Schober's series formula for the area of the Mandelbrot set in the case where $m\equiv 2 \mod 4$.
       \end{abstract}

\maketitle

\begin{section}{Introduction}

The Mandelbrot set $M$ is defined as the set of complex numbers $c \in \mathbb{C}$ for which the sequence $\{z_{n}\}$ defined by the recursion
\begin{equation} \label{eq:quadratic}
z_{n} = z_{n-1}^2 + c
\end{equation}
\noindent  with initial value \(z_{0} = 0\) remains bounded for all $n \geq 0$.  Douady and Hubbard \cite{DH} proved that $M$ is connected and Shishikura \cite{S} proved that $M$ has fractal boundary of Hausdorff dimension 2.  However, it is unknown whether the boundary of $M$ has positive Lebesgue measure, although Julia sets with positive area are known to exist (Buff and Ch\'eritat \cite{BC}). 

Ewing and Schober \cite{ES2} derived a series formula for the area of $M$ by considering its complement, $\tilde{M}$, inside the Riemann sphere $\overline{\mathbb{C}}=\mathbb{C}\cup \{\infty\}$, i.e. $\tilde{M}=\overline{\mathbb{C}}-M$.  It is known that $\tilde{M}$ is simply connected with mapping radius 1 (\cite{DH}).  In other words, there exists an analytic homeomorphism 
\begin{equation}\label{eq: analytichom}
	\psi(z) = z + \sum_{m = 0}^{\infty}{b_{m}z^{-m}}
\end{equation}
which maps the domain $\Delta = \left\{z: 1 < |z| \leq \infty \right\} \subset \overline{\mathbb{C}}$ onto $\tilde{M}$.  
It follows from the classic result of Gronwall \cite{Gr} that the area of the Mandelbrot set $M=\overline{\mathbb{C}}-\tilde{M}$ is given by

\begin{equation} \label{eq:area}
A = \pi \left[1 - \sum_{m = 1}^{\infty}{m|b_{m}|^2} \right].
\end{equation}

The arithmetic properties of the coefficients $b_m$ have been studied in depth, first by Jungreis \cite{J}, then independently by Levin \cite{L, L2}, Bielefeld, Fisher, and Haeseler \cite{BFH}, Ewing and Schober \cite{ES1, ES2}, and more recently by Shimauchi \cite{Sh}.   In particular, Ewing and Schober \cite{ES2} proved the following formula for the coefficients $b_m$.

\begin{theorem}[Ewing-Schober \cite{ES2}] Suppose $m\leq 2^{n+1}-3$.  Define the set of $n$-tuples
     \[J = \{\mathbf{j}=(j_1,\ldots,j_n): (2^n-1)j_1+\ldots + (2^2-1)j_{n-1}+(2-1)j_n = m+1\}\] 
     and given any $\mathbf{j} \in J$, set
     \[\alpha_{\mathbf{j}}(k):=\alpha(k):=\alpha=\frac{m}{2^{n-k+1}}-2^{k-1}j_1 - 2^{k-2}j_2 - \ldots - 2 j_{k-1}.\] Then
     \begin{equation} \label{eq:ewing-schober}
     b_m = -\frac{1}{m} \sum_{J} \prod_{k=1}^{n}C_{j_k}(\alpha(k))
     \end{equation}
     where $C_{j_k}(\alpha(k))$ is the binomial coefficient
     \begin{equation}
C_{j_k}(\alpha(k))=\frac{\alpha(\alpha-1)(\alpha-2)\cdots (\alpha-(j_k-1))}{j_k!}.
     \end{equation}
\end{theorem}
Using formula (\ref{eq:ewing-schober}) to compute $b_m$ is impractical as it requires determining the set of tuples $J$, which is computationally hard.  However, since it is known that each $b_m$ is rational and has denominator equal to a power of 2, it is then useful to find a formula for its 2-adic valuation.  Towards this end, Levin \cite{L} gave such a formula when $m$ is odd, and Shimauchi \cite{S} established an upper bound valid for all $m$ with equality if and only if $m$ is odd.  

\begin{definition} Let $n$ be a non-negative integer.  We define

\noindent (a) $\nu(n)$ to be the 2-adic valuation of $n$.

\noindent (b) $s(n)$ (called the sum-of-digits function) to be the sum of the binary digits of $n$.

\end{definition}

\begin{theorem}[Levin \cite{L}, Shimauchi \cite{S}]
Let $m$ be a non-negative integer.  Then
\begin{equation}\label{eq:shimauchi}
-\nu(b_m)\leq 2(m+1)-s(2(m+1))
\end{equation}
Moreover, equality holds precisely when $m$ is odd.
\end{theorem}

In this paper we prove Zagier's conjecture (see \cite{BFH}) regarding a formula for the 2-adic valuation of $b_m$ when $m\equiv 2 \mod 4$.  
\begin{theorem}[Zagier's Conjecture \cite{BFH}] \label{th:zagier}
Suppose $m\equiv 2 \mod 4$. Then
\begin{equation}
    -\nu(b_{m}) = \Bigl\lfloor \frac{2}{3}(m+1)\Bigr\rfloor - s\left(\Bigl\lfloor\frac{2}{3}(m+1)\Bigr\rfloor\right) +\epsilon(m),
\end{equation}
where
\begin{equation}
    \epsilon(m) = 
    \begin{cases}
        0, & \mbox{if } m \equiv 22 \mod 24;\\
        1, & \mbox{otherwise}.
    \end{cases}
\end{equation}
\end{theorem}

Our proof relies on determining those tuples $\mathbf{j}_{\max}\in J$ that maximize $V(\mathbf{j}):=-\nu\left(\prod_{k=1}^{n}C_{j_k}(\alpha(k))\right)$, i.e., $V(\mathbf{j}) < V(\mathbf{j}_{\max})$ for all $\mathbf{j} \in J$.  In particular, we show for $m\equiv 2 \mod 4$ that this largest 2-adic valuation  $V(\mathbf{j}_{\max})$ is achieved by exactly one tuple $\mathbf{j}_{\max}$ or else by exactly three tuples $\mathbf{j}_{\max}$, $\mathbf{j}'_{\max}$, $\mathbf{j}''_{\max}$ in the special case where $m\equiv 22 \mod 24$.  To prove that $V(\mathbf{j}) < V(\mathbf{j}_{\max})$ for all $\mathbf{j} \in J$, we derive lemmas to compare the values of $V(\mathbf{j})$ for different types of tuples.  For example,  if $m=38$, then it holds that
\[
V((2,1,0,2)) < V((0,5,1,1)) < V((0,0,13,0)),
\]
where $\mathbf{j}_{\max}=(0,0,13,0)$.
We refer to the chain of tuples
\[
(2,1,0,2) \rightarrow (0,5,1,1) \rightarrow \mathbf{j}_{\max}
\]
 as a set of tuple transformations.

As a result of our comparison lemmas (derived in Sections 2 and 3), we have the result
\begin{equation}
-\nu(b_m)=1+ V(\mathbf{j}_{\max}).
\end{equation}
This follows from the fact that the 2-adic valuation of the sum of any number of fractions (whose denominators are powers of 2 and whose numerators are odd) is equal to the largest 2-adic valuation of all the fractions, assuming that there are an odd number of fractions with the same largest 2-adic valuation.  It remains to calculate $V(\mathbf{j}_{\max})$ in each case, which then establishes Zagier's conjecture.

\end{section}

\begin{section}{Tuple Transformations}

We begin with preliminary definitions.

\begin{definition}\label{de:beta}
    Given $\mathbf{j}\in J$, define 
    \[\beta_{\mathbf{j}}(k):=\beta(k):=\beta=2^{n-k+1}\alpha(k)=m-2^n j_1 - 2^{n-1} j_2 - \cdots - 2^{n-k+2} j_{k-1}\] 
    and 
    \[B(k)=\beta(\beta-2^{n-k+1})(\beta-2\cdot 2^{n-k+1})\cdots (\beta-(j_k-1)\cdot 2^{n-k+1}).\]
\end{definition}

\begin{lemma} We have
\[\nu(B(k))=j_k\]
 for $1\leq k \leq n-\nu(m)$.
\end{lemma}

\begin{proof}
    First, we establish that $\nu(\beta(k)) = \nu(m)$ for  $1 \leq k \leq n -\nu(m)$. This follows from
    \begin{align*}
        \nu(\beta) &= \nu(m-2^n j_1 - 2^{n-1} j_2 - \cdots - 2^{n-k+2} j_{k-1})\\
        &= \nu(m - (2^n j_1 - 2^{n-1} j_2 - \cdots - 2^{n-k+2} j_{k-1}))\\
        &= \nu(m),
    \end{align*}
which holds since $\nu(2^n j_1 - 2^{n-1} j_2 - \cdots - 2^{n-k+2} j_{k-1}) \geq n-k+2 > \nu(m)$. Then by definition we have
    \begin{equation*}
        B(k) = \beta(\beta-d^{n-k+1})(\beta-2d^{n-k+1})\cdots (\beta-(j_k-1)d^{n-k+1})
    \end{equation*}
    Taking the 2-adic valuation of both sides and expanding the right-hand side gives
    \begin{align*}
        \nu(B(k)) &= \nu(\beta(\beta-2^{n-k+1})(\beta-2\cdot 2^{n-k+1})\cdots (\beta-(j_k-1)2^{n-k+1}))\\
        &= \nu(\beta) + \nu(\beta-2^{n-k+1}) + \nu(\beta-2\cdot 2^{n-k+1}) + \cdots + \nu(\beta-(j_k-1)2^{n-k+1})
    \end{align*}
    Since $n-k+1>\nu(m)$, $\nu(\beta - p\cdot 2^{n-k+1}) = 1$ for all integers $p$. Thus \[\nu(\beta) + \nu(\beta-2^{n-k+1}) + \nu(\beta-2(2^{n-k+1})) + \cdots + \nu(\beta-(j_k-1)2^{n-k+1}) = 1 + 1 + \dots + 1\] where there are $j_k$ 1's.  Thus,
    $\nu(B(k)) = j_k$ as desired.
\end{proof}

\begin{lemma}\label{le:2-adic-case-1} We have
\begin{equation}
-\nu(C_{j_k}(\alpha(k))= (n-k+1)j_k - s(j_k)
\end{equation}
for $1\leq k \leq n-\nu(m)$
\end{lemma}
\begin{proof}
 It is clear from Definition \ref{de:beta} that
\begin{align*}
    C_{j_k}(\alpha(k)) &= \frac{\alpha(\alpha-1)(\alpha-2)\cdots (\alpha-(j_k-1))}{j_k!} \\
    &= \frac{\beta(\beta-2^{n-k+1})(\beta-2\cdot 2^{n-k+1})\cdots (\beta-(j_k-1)2^{n-k+1})}{2^{j_k(n-k+1)}j_k!} \\
    &= \frac{B(k)}{2^{j_k(n-k+1)}j_k!}
\end{align*}
and thus
\begin{align*}
-\nu(C_{j_k}(\alpha(k))    &= -\nu\left(\frac{B(k)}{2^{j_k(n-k+1)}j_k!}\right) \\
    &= -(\nu(B(k) - \nu(2^{j_k(n-k+1)}j_k!)) \\
    &= (n-k+1)j_k + j_k - s(j_k) - \nu(B(k)) \\
    &= (n-k+1)j_k - s(j_k)
\end{align*}
since we have from Lemma \ref{le:2-adic-case-1} that $\nu(B(k)) = j_k$ for $1 \leq k \leq n - \nu(m)$. 
\end{proof}

We now consider the case where $k > n - \nu(m)$.  Define $c(x,y)$ to be the number of carries performed when summing two non-negative integers $x$ and $y$ in binary.  It is a well known result that 
\[
c(x,y) = s(x)+s(y)-s(x+y).
\]

\begin{lemma} \label{le:2-adic-case-2}
    Let $\mathbf{j}\in J$.  Then for $k > n - \nu(m)$, we have
    \begin{equation}
        -\nu(C_{j_k}(\alpha(k))) =
        \begin{cases}
            -c(j_k,-\alpha(k)-1), & \alpha(k) <0; \\
            -\infty, & 0\leq \alpha(k)\leq j_k; \\
            c(j_k, \alpha(k)-j_k), & \alpha(k)>j_k.
        \end{cases}
    \end{equation}
\end{lemma}

\begin{proof}
First, we demonstrate that $\alpha(k)$ is an intger when $k>n-\nu(m)$. By definition, we have
    
\[\alpha(k) = \frac{m}{2^{n-k+1}}-2^{k-1}j_1 - 2^{k-2}j_2 - \ldots - 2 j_{k-1}.\] 
Since $\nu(m) \geq n-k+1$, it follows that $m$ is divisible by $2^{n-k+1}$. Thus, $\frac{m}{2^{n-k+1}}$ is an integer, and since the remaining terms are all integers, $\alpha(k)$ must be an integer as well.  

If $\alpha(k) < 0$, we have
    \begin{align*}
        -\nu(C_{j_k}(\alpha(k))) &= -\nu\left(\frac{\alpha(\alpha - 1)\dots(\alpha - j_k + 1)}{j_k!}\right)\\
        &= j_k - s(j_k) - \nu((\alpha - j_k + 1) \dots (\alpha - 1)\alpha)\\
        &= j_k - s(j_k) - (\nu((-\alpha - j-k + 1)!) - \nu(-\alpha - 1)) \\
        &= -s(j_k) - s(-\alpha - 1) + s(-\alpha-1+j_k)\\
        &= -c(j_k, -\alpha - 1).
    \end{align*}
    On the other hand, if $0 \leq \alpha(k)$, then $C_{j_k}(\alpha(k)) = 0$, and therefore $\nu(C_{j_k}) = \infty$. Lastly, if $\alpha(k) > j_k$, then we have
    \begin{align*}
        -\nu(C_{j_k}(\alpha)) &= -\nu\left(\frac{\alpha!}{(\alpha-j_k)!j_k!}\right)\\
        &= \alpha - s(\alpha) - (\alpha-j_k) + s(\alpha-j_k) - j_k + s(j_k)  \\
        &= s(j_k) + s(\alpha-j_k) - s(\alpha) \\
        &= c(j_k, \alpha(k)-j_k)
    \end{align*}
    as desired.
\end{proof}

\begin{definition} \label{de:V}
    For convenience, define
    \begin{equation}
        \gamma(m,k):=\gamma(k) = 
        \begin{cases}
            -c(j_k,-\alpha(k)-1), & \alpha(k) <0; \\
            \infty, & 0\leq \alpha(k)\leq j_k; \\
            c(j_k, \alpha(k)-j_k), & \alpha(k)>j_k,
        \end{cases}
    \end{equation}
    and for any tuple $\mathbf{j}\in J$, define
    \begin{equation}
    v(m, \mathbf{j}) = \sum_{k=1}^{n-\nu(m)} [(n-k+1)j_k - s(j_k)]
    \end{equation}
    and
    \begin{equation}
    V(m,\mathbf{j}):=V(\mathbf{j}) = -\nu\left(\prod_{k=1}^{n}C_{j_k}(\alpha(k))\right).
    \end{equation}
    In the case where $m\equiv 2 \mod 4$ so that $\nu(m)=1$, we shall simply write 
    \begin{equation}
     v(\mathbf{j}):=v(m, \mathbf{j}) = \sum_{k=1}^{n-1} [(n-k+1)j_k - s(j_k)].
    \end{equation}
    
\end{definition}
\end{section}

The next lemma follows immediately from Definition \ref{de:V} and Lemmas \ref{le:2-adic-case-1} and \ref{le:2-adic-case-2}.
\begin{lemma}\label{le:V(j)}
    We have 
    \[V(\mathbf{j}) =v(m,\mathbf{j}) + \sum_{k=n-\nu(m)+1}^n \gamma(k)\]
    and in particular if $m\equiv 2 \mod 4$, then
    \begin{equation}
    V(\mathbf{j}) = v(\mathbf{j}) + \gamma(n).
    \end{equation}
\end{lemma}

We now consider tuple transformations that allow us to compare $v(m,\mathbf{j})$ for different types of tuples.

\begin{lemma} \label{le:tuple-transformation-1}
    Suppose $\nu(m)\geq 1$.  Let $\mathbf{j}$ be a J-tuple and $i < n-\nu(m)$ be such that $j_i \neq 0$.  Define the tuple $\mathbf{j}'=(j_1',\ldots,j_n')$ by
    \begin{align*}
    j_k' = 
    \begin{cases}
    j_k, & k \neq i, i+1, n  \\ 
    j_i - r, & k=i \\
    j_{i+1}+p, & k=i+1 \\
    j_n+q, & k=n
    \end{cases}
    \end{align*}
     where $r$ is the largest power of 2 less than $j_i$, and $p$ and $q$ satisfy 
     \begin{equation}\label{eq:pq}
     (2^{n-i}-1)p + q = (2^{n-i+1}-1)r
     \end{equation}
      with $q < 2^{n-i}-1$.  Then
    \[
    v(m,\mathbf{j})<v(m,\mathbf{j}').
    \]
\end{lemma}
\begin{proof}
    It is clear that $p$ and $q$ exist by Euclid's Division Theorem.  Then since $j_k=j_k'$ for all $k\neq i, i+1, n$, the corresponding terms will cancel when we compute the difference $v(\mathbf{j}')-v(\mathbf{j})$. If $i<n-2$, then
    
    \begin{align*}
        v(m,\mathbf{j}') - v(m,\mathbf{j}) &= (n-i)p - (n-i+1)r + s(j_i) - s(j_i - r) + s(j_{i+1}) - s(j_{i+1} + p)\\
        & \geq (n-i)p - (n-i+1)r + 1 -s(p)\\
        & > \frac{n-i-1}{2}p - \frac{n-i+1}{2} - \lceil \log_2(p) \rceil \\
        &\geq 0
    \end{align*}
    since $r<(p+1)/2$ and $p \geq 2$.
    The remaining case, $i =n- 2$, can be easily proven by similar means.
\end{proof}

Observe that we can apply Lemma \ref{le:tuple-transformation-1} repeatedly to transform any tuple $\mathbf{j}\in J$ containing a non-zero element $j_i$, $1\leq i \leq n-\nu(m)$, to a tuple $\mathbf{j}'\in J$ with $j_i'=0$.  Thus, any tuple $\mathbf{j} \in J$ can be transformed to a tuple $\mathbf{j}'$, where all elements $j_i'=0$ except for $i\geq n-\nu(m)$, with $v(\mathbf{j}) < v(\mathbf{j}')$.  We will make use of this fact later on.

\begin{lemma}\label{le:tuple-transformation-2}
    Let $\mathbf{j}$ be a J-tuple where $j_n > 2$, and $\mathbf{j}'$ be the tuple such that
    
        \begin{align*}
    j_k' = 
    \begin{cases}
    j_k, & 1\leq k \leq n-\nu(m)-1; \\ 
    j_{n-\nu(m)}+p, & k=n-\nu(m); \\
    0, & n-\nu(m) < k < n; \\
    \sum_{k=n-\nu(m)+1}^n (2^{n-k+1}-1)j_k - (2^{\nu(m)+1} - 1)p, & k=n,
    \end{cases}
    \end{align*}
    where $p$ is chosen to be as largest as possible so that $j_n' < 2^{\nu(m)+1} -1$. Then
    \[
    v(m,\mathbf{j}) < v(m,\mathbf{j}').
    \]
\end{lemma}
\begin{proof}
    We have that
    \begin{align*}
        v(m,\mathbf{j}') - v(m,\mathbf{j}) &= (n-\nu(m)+1)(j_{n-\nu(m)} + p) - s(j_{n-\nu(m)} + p) \\
        & \quad \quad - (n-\nu(m)+1)j_{n-\nu(m)} + s(j_{n-\nu(m)}) \\
        &= (n-\nu(m)+1)p + s(j_{n-\nu(m)}) - s(j_{n-\nu(m)} + p) \\
        & = (n-\nu(m)+1)p + c(j_{n-\nu(m)},p) - s(p) \\
        &\geq (n-\nu(m)+1)p - s(p) \\
        &> 0.
    \end{align*}
\end{proof}

In particular, when $m \equiv 2 \mod 4$, Lemma \ref{le:tuple-transformation-2} allows us to transform a tuple $\mathbf{j}\in J$, whose elements are all zero except for $j_{n-1}$ and $j_n>2$, to a tuple $\mathbf{j}'\in J$, whose elements are also all zero but with $j_n' \leq 2$, so that $v(\mathbf{j})<v(\mathbf{j}')$.

\begin{section}{Zagier's Conjecture}

In this section we prove Zagier's conjecture for the case where $m\equiv 2 \mod 4$, which we assume throughout this section.  In order to do this, we first derive additional lemmas that allow us to compare $V(\mathbf{j})$ for the tuple transformations described in the previous section.

\begin{lemma}
    If $m + 1 \equiv 0 \mod 3$, then $V(\mathbf{j}) < V(\mathbf{j}')$ for all $\mathbf{j} \neq \mathbf{j}'$, where $\mathbf{j}' = (0, 0 , \dots, \frac{m+1}{3}, 0)$.
\end{lemma}
\begin{proof}
By Lemmas \ref{le:tuple-transformation-1} and \ref{le:tuple-transformation-2}, we can transform $\mathbf{j}$ to a tuple $\mathbf{j}'$ so that $j_i'=0$ for all $i< n-1$ since $\nu(m)=1$.  Moreover, $j_{n-1}'=(m+1)/3$ and $j_n'=0$ since $m+1\equiv 0 \mod 3$.  It follows that
\begin{align*}
    V(\mathbf{j}) &= \sum_{k=1}^{n-1} [(n-k+1)j_k - s(j_k)] - c(j_n,-\alpha(n)-1) \\
    &\leq \sum_{k=1}^{n-1} [(n-k+1)j_k - s(j_k)]] = v(\mathbf{j}) \\
    &< v(\mathbf{j}') = V(\mathbf{j}')
\end{align*}
since $c(j_n',-\alpha'(n)-1)=0$ due to Lemma \ref{le:V(j)}.
\end{proof}

\begin{lemma}
    If $m + 1 \equiv 1 \mod 3$, then $V(\mathbf{j}) < V(\mathbf{j}')$ for all $\mathbf{j} \neq \mathbf{j}'$, where $\mathbf{j}' = (0, 0 , \dots, \frac{m}{3}, 1)$.
\end{lemma}
\begin{proof}
We have $V(\mathbf{j}) < V(\mathbf{j}')$
by the same reasoning as in the previous lemma.
\end{proof}

\begin{lemma}
    If $m+1 \equiv 2 \mod 3$ and $m \equiv 2 \mod 8$, then $V(\mathbf{j}) < V(\mathbf{j}')$ for all $\mathbf{j} \neq \mathbf{j}'$, where $\mathbf{j}' = (0, 0 , \dots, \frac{m-1}{3}, 2)$.
\end{lemma}
\begin{proof}
Again, such a tuple $\mathbf{j}'$ exists because of Lemmas \ref{le:tuple-transformation-1} and \ref{le:tuple-transformation-2}.  We first determine the binary representation of $-\alpha(n)-1$.  Since
\begin{align*}
    -\alpha(n) - 1 &= - \frac{j_n - (1 + j_1 + \dots + j_{n-1})}{2} - 1 \\
    &= - \frac{1 - \frac{m-1}{3}}{2} - 1 \\
    &= \frac{\frac{m-1}{3} - 1}{2} - 1 \\
    &= \frac{\frac{m-1}{3}-3}{2} \\
    &= \frac{m-10}{6}
\end{align*}
and $m \equiv 2 \mod 8$ by assumption, it follows that $-\alpha(n) - 1$ has binary representation $b_n\cdots b_3100$. It follows that $c(2, -\alpha(n) -1) = 0$ and thus $V(\mathbf{j}')=v(\mathbf{j}')$ by Lemma \ref{le:V(j)}.  Moreover, we have
\begin{align*}
    V(\mathbf{j}') &= \sum_{k=1}^{n-1} [(n-k+1)j_k - s(j_k)] - c(j_n,-\alpha(n)-1) \\
    &= \frac{2(m-1)}{3} - s\left(\frac{m-1}{3}\right) - c(2, -\alpha(n) -1) \\
    &= \frac{2(m-1)}{3} - s\left(\frac{2(m-1)}{3}\right).
\end{align*}
It remains to be shown that $V(\mathbf{j}) < V(\mathbf{j}')$ for all $\mathbf{j} \neq \mathbf{j}'$.  This follows from
\begin{align*}
    V(\mathbf{j}) &= \sum_{k=1}^{n-1} [(n-k+1)j_k - s(j_k)] - c(j_n,-\alpha(n)-1) \\
    &\leq \sum_{k=1}^{n-1} [(n-k+1)j_k - s(j_k)]] =v(\mathbf{j}) \\
    &< v(\mathbf{j}') = V(\mathbf{j}').
\end{align*}
This proves the lemma.
\end{proof}

In order to handle the case $m + 1 \equiv 2 \mod 3$ and $m \equiv 6 \mod 8$ (or equivalently $m \equiv 22 \mod 24$), we will need the following lemma.  First, we define the following three special tuples, which exist for this case:
\begin{align*}
\mathbf{j}' & = (0, 0, \dots, \frac{m-1}{3}, 2) \\
\mathbf{j}'' & = (0, 0, \dots, \frac{m-1}{3} - 1, 5)\\
\mathbf{j}''' & = (0 , 0, \dots, 1, \frac{m-1}{3} - 2, 1).
\end{align*}

\begin{lemma} \label{le:6-mod-8}
Suppose $m + 1 \equiv 2 \mod 3$ and $m \equiv 6 \mod 8$.  Then for all $\mathbf{j} \notin \{\mathbf{j}', \mathbf{j}'', \mathbf{j}'''\}$, we have
\[
V(\mathbf{j}) < V(\mathbf{j}''').
\]
\end{lemma}
\begin{proof}
    Since $\alpha_{\mathbf{j}'''}(n)$ is odd and $j_n'''=1$, we have $c(j_n''',-\alpha(n)-1)=0$ and thus $V(\mathbf{j})=v(\mathbf{j})$.  Moreover, we have
    \begin{align*}
        V(\mathbf{j}''') &= \sum_{k=1}^{n-1} [(n-k+1)j_k''' - s(j_k''')] - c(j_n''',-\alpha(n)-1) \\
        &= 3\cdot 1 - s(1) + \frac{2(m-7)}{3} - s\left(\frac{m-7}{3}\right)\\
        &= \frac{2(m-1)}{3}  -2 - s\left(\frac{m-1}{3} - 2\right) \\
        &= \frac{2(m-1)}{3} - s\left(\frac{m-1}{3}\right) - 1.
    \end{align*}
    Thus, it suffices to show that $v(\mathbf{j}) < v(\mathbf{j}''')$ since this will imply $V(\mathbf{j}) \leq v(\mathbf{j}) < v(\mathbf{j}')=V(\mathbf{j}')$. Note that for any tuple $\mathbf{j}$ containing an element $j_i \neq 0$ such that $1 \leq i \leq n-3$, we have $v(\mathbf{j}) < v(\mathbf{g})$ for some tuple $\mathbf{g}$ with $g_i = 0$ for all $1 \leq i \leq n-3$ and $g_{n-2} = 2^k$ for some $k$. To construct such a tuple $\mathbf{g}$, we simply apply the tuple transformation in Lemma \ref{le:tuple-transformation-1} repeatedly. 
    
    We now consider 3 cases. First, if $\mathbf{g} = \mathbf{j}'''$, then the theorem holds trivially. If $g_{n-2} > 1$, we proceed in two steps. Let $7(g_{n-2} - 1) = 3p + q$ where $q < 3$, and let $\mathbf{g}'$ be such that
    \begin{align*}
        g_i' &= 0 \ \text{for} \ 1 \leq i \leq n-3 \\
        g_{n-2}' &= 1 \\
        g_{n-1}' &= g_{n-1} + p \\
        g_n' &= g_n + q.
    \end{align*}
    Then we have
    \begin{align*}
        v(\mathbf{g}') - v(\mathbf{g}) &= 2 + 2(g_{n-1} + p) - s(g_{n+1} + p) - 3 g_{n-2} + 1 - 2g_{n-1} + s(g_{n-1}) \\
        &\geq 2p - 3g_{n-2} - \lceil \log_2(p) \rceil + 3 \\
        &\geq \frac{11p}{7} - \lceil \log_2(p) \rceil + \frac{12}{7} \\
        &> 0.
    \end{align*}
    Then applying Lemma \ref{le:tuple-transformation-2}  to $\mathbf{g}'$ completes the proof for this case. If $g_{n-2} = 0$, then we proceed as follows. Let $\frac{m-1}{3} = g_{n-1} + p$. Note that because $\mathbf{g} \notin \{\mathbf{j}', \mathbf{j}'', \mathbf{j}'''\}$, we have $p \geq 2$.  Thus,
    \begin{align*}
        v(\mathbf{j}''') - v(\mathbf{g}) &= 2(g_{n-1} + p) - s(g_{n-1} + p) - 1 - 2g_{n-1} + s(g_{n-1}) \\
        &\geq 2p - 1 - s(p) \\
        &> 0.
    \end{align*}
    This completes the proof.
\end{proof}

\begin{lemma}
    If $m + 1 \equiv 2 \mod 3$ and $m \equiv 46 \mod 48$, then $V(\mathbf{j}) < V(\mathbf{j}''')$ for all $\mathbf{j} \neq \mathbf{j}'$, where $\mathbf{j}''' = (0, 0 , \dots, 1, \frac{m-7}{3}, 1)$.
\end{lemma}
\begin{proof}
    In light of Lemma \ref{le:6-mod-8}, it suffices to prove that $V(\mathbf{j}')<V(\mathbf{j}''')$ and $V(\mathbf{j}'')<V(\mathbf{j}''')$.  We first consider $\mathbf{j}'$.  We have 
    \begin{align*}
\alpha_{\mathbf{j}''}(n) & = m/2-2^{n-1}j_1 - 2^{n-2}j_2 - \cdots - 2j_{n-1} \\
& = m/2 - 2(m-1)/3  = (4-m)/6,
\end{align*}
    which implies $-\alpha(n)-1=(m-10)/6$ has binary expansion $b_n\ldots b_3110$.  Thus, $c(j_n',-\alpha(n)-1)>0$ since $j_n'=2$.  It follows that
    \begin{align*}
        V(\mathbf{j}') &= \sum_{k=1}^{n-1} [(n-k+1)j_k' - s(j_k')] - c(j_n',-\alpha(n)-1) \\
        & < \sum_{k=1}^{n-1} [(n-k+1)j_k' - s(j_k')]  \\
        & =  \frac{2(m-1)}{3} - s\left(\frac{m-1}{3}\right) \\
        &= \frac{2(m-1)}{3} - 2 - s\left(\frac{m-1}{3} - 2\right) \\
        &= \frac{2(m-1)}{3} - s\left(\frac{m-1}{3}\right) - 1 \\
        & = V(\mathbf{j}''').
    \end{align*}

As for $\mathbf{j}''$, we have $c(j_n'',-\alpha(n)-1)>0$ since $j_n''=5$ and 
\begin{align*}
\alpha_{\mathbf{j}''}(n) & = m/2-2^{n-1}j_1 - 2^{n-2}j_2 - \cdots - 2j_{n-1} \\
& = m/2 - 2(m-4)/3  = (16-m)/6,
\end{align*}
which implies $-\alpha(n)-1=(m-22)/6$ has binary expansion $b_n\ldots b_3100$.
 It follows that
    \begin{align*}
        V(\mathbf{j}') &= \sum_{k=1}^{n-1} [(n-k+1)j_k' - s(j_k')] - c(j_n',-\alpha(n)-1) \\
        & < \sum_{k=1}^{n-1} [(n-k+1)j_k' - s(j_k')]  \\
        & =  \frac{2(m-4)}{3} - s\left(\frac{m-4}{3}\right) \\
        &= \frac{2(m-1)}{3} - 2 - s\left(\frac{m-1}{3} - 1\right) \\
        &= \frac{2(m-1)}{3} - s\left(\frac{m-1}{3}\right) - 1 \\
        & = V(\mathbf{j}''').
    \end{align*}
This completes the proof.

\end{proof}

\begin{lemma}
    If $m + 1 \equiv 2 \mod 3$ and $m \equiv 22 \mod 48$, then
    \[
    V(\mathbf{j}) < V(\mathbf{j}')
    \]
 for all $\mathbf{j} \notin \{\mathbf{j}', \mathbf{j}'', \mathbf{j}'''\}$.  Moreover,   
\[
V(\mathbf{j}') = V(\mathbf{j}'') = V(\mathbf{j}''') = \frac{2(m-1)}{3} - s\left(\frac{2(m-1)}{3}\right) - 1.
\]
\end{lemma}

\begin{proof}
Again, in light of Lemma \ref{le:6-mod-8}, it suffices to prove that $V(\mathbf{j}') = V(\mathbf{j}'') = V(\mathbf{j}''')$.  Write $m=48q+22$ for $q\in \mathbb{N}$ and so that the elements of $\mathbf{j}'$, $\mathbf{j}''$, and $\mathbf{j}'''$ take the form
\begin{equation}
j_i'=\begin{cases}
0, & 1\leq i \leq n-3 \\
1, & i=n-2 \\
16q+5 & i=n-1 \\
1 & i=n,
\end{cases}
\end{equation}
\begin{equation}
j_i''=\begin{cases}
0, & 1\leq i \leq n-3 \\
0, & i=n-2 \\
16q+7 & i=n-1 \\
2 & i=n,
\end{cases}
\end{equation}
and
\begin{equation}
j_i'''=\begin{cases}
0, & 1\leq i \leq n-3 \\
0, & i=n-2 \\
16q+6 & i=n-1 \\
5 & i=n,
\end{cases}
\end{equation}

It is straightforward to show that
\begin{align*}
\alpha_{\mathbf{j}}(n) & =-(8q+3)<0 \\
\alpha_{\mathbf{j}'}(n) & =-(8q+3)<0 \\
\alpha_{\mathbf{j}''}(n) & =-(8q+1)<0.
\end{align*}
Then
\begin{align*}
V(\mathbf{j}') & =3j_{n-2}'-s(j_{n-2}')+2j_{n-1}'-s(j_{n-1}')-c(j_n,-\alpha_{\mathbf{j}'}(n)-1) \\
& = 3(1)-s(1)+2(16q+5)-s(16q+5)-c(1,8q+2) \\
& = 32q+12-s(q)-s(5) \\
& = 32q+10-s(q).
\end{align*}
Similarly,
\begin{align*}
V(\mathbf{j}'') & =3j''_{n-2}-s(j''_{n-2})+2j''_{n-1}-s(j''_{n-1})-c(j''_n,-\alpha_{\mathbf{j}''}(n)-1) \\
& = 3(0)-s(0)+2(16q+7)-s(16q+7)-c(2,8q+2) \\
& = 32q+14-s(q)-s(7)-c(2,8q+2) \\
& = 32q+10-s(q)
\end{align*}
and
\begin{align*}
V(\mathbf{j}''') & =3j'''_{n-2}-s(j'''_{n-2})+2j'''_{n-1}-s(j'''_{n-1})-c(j'''_n,-\alpha_{\mathbf{j}'''}(n)-1) \\
& = 3(0)-s(0)+2(16q+6)-s(16q+6)-c(5,8q) \\
& = 32q+12-s(q)-s(6)-c(5,8q) \\
& = 32q+10-s(q).
\end{align*}
Thus, $V(\mathbf{j}') = V(\mathbf{j}'')=V(\mathbf{j}''')$.
\end{proof}

The following theorem summarizes the form of the maximum tuple $\mathbf{j}_{\max}$ for the case $m\equiv 2 \mod 4$.

\begin{theorem} Suppose $m\equiv 2 \mod 4$.  The maximum tuple $\mathbf{j}_{\max}$ occurs in the following form:
\begin{enumerate}
\item If $m+1\equiv 0 \mod 3$, then $\mathbf{j}_{\max}=(0,...,0,p,0)$ where $p=(m+1)/3$.
\item if $m+1\equiv 1 \mod 3$, then $\mathbf{j}_{\max}=(0,...,0,p,1)$ where $p=m/3$.
\item If $m+1\equiv 2 \mod 3$ and
\begin{enumerate}
\item If $m\equiv 2 \mod 8$, then $\mathbf{j}_{\max}=(0,...,0,p,2)$ where $p=(m-1)/3$.
\item If $m\equiv 46 \mod 48$, then $\mathbf{j}_{\max}=(0,...,1,p-2,1)$ where $p=(m-1)/3$.
\item If $m\equiv 22 \mod 48$, then 
\begin{center}
$\mathbf{j}_{\max}=(0,...,1,(m-7)/3,1), (0,...,0,(m-1)/3,2), (0,\ldots,0, (m-4)/3,5)$.
\end{center}
\end{enumerate}
\end{enumerate}

\end{theorem}

We now have all the necessary ingredients to prove Zagier's conjecture.

\begin{proof}[Proof of  Theorem \ref{th:zagier} (Zagier's Conjecture):]
We divide the proof into the following cases: 

\begin{enumerate}
\item $m+1\equiv 0 \mod 3$.
\item $m+1\equiv 1 \mod 3$.
\item $m+1\equiv 2 \mod 3$ and 
\begin{enumerate}
\item $m\equiv 2 \mod 8$.
\item $m\equiv 46 \mod 48$.
\item $m\equiv 22 \mod 48$.
\end{enumerate}
\end{enumerate}

\noindent Case (1): Write $m+1=3p$ for some positive integer $p$.  Since $m \equiv 2 \mod 4$, it follows that $3p-1 \equiv 2 \mod 4$ and so $p \equiv 1 \mod 4$.  Now, recall that $\mathbf{j}_{\max}=(j_{n-1},j_n)=(p,0)$, we have $\alpha(n)=-(1+p)/2$.  Then using the relation 
\[
c(j_n,-\alpha(n)-1)=s_(j_n) + s(-\alpha(n)-1) - s(j_n-\alpha(n)-1),
\]
we have
\begin{align*}
    -\nu(b_{2,m}) 
    & = \nu(m)+ \sum_{k=1}^{n-1} [(n-k+1)j_k - s(j_k)]- s(j_n)- s(-\alpha(n)-1) + s(j_n-\alpha(n)-1)  \\
    &= 1+2j_{n-1}-s(j_{n-1}) \\
    & = 1+2p-s(p) \\
    & = 1+2p-s(2p)\\
    & = 1+\lfloor 2p\rfloor - s(\lfloor 2p\rfloor) \\
    & = \epsilon(m)+\left\lfloor \frac{2}{3}(m+1)\right\rfloor -s\left(\left\lfloor \frac{2}{3}(m+1)\right\rfloor\right).
\end{align*}

\noindent Case (2): Write $m+1=3p+1$ for some positive integer $p$.  Since $m \equiv 2 \mod 4$, it follows that $3p \equiv 2 \mod 4$ and so $p \equiv 2 \mod 4$.  
Since in this case $\mathbf{j}_{\max}=(j_{n-1},j_n)=(p,1)$, we have $\alpha(n)=-p/2$. It follows that 
\begin{align*}
    -\nu(b_{2,m}) &= \nu(m)+ \sum_{k=n-1}^{n-1} [(n-k+1)j_k - s(j_k)] - s(j_n) - s(-\alpha(n)-1)+ s(j_n-\alpha(n)-1)  \\
    &= 1+2j_{n-1}-s(j_{n-1})-s(j_n)-s(p/2-1) + s(j_n+p/2-1) \\
    & = 1+2p-s(p) -1 - s((p-2)/2)+s(p/2) \\
    & = 2p-s(p-2) \\
    & = 1+2p-s(p) \\
    & = 1+2p-s(2p)\\
    & = 1+\lfloor 2p+2/3\rfloor - s(\lfloor 2p+2/3\rfloor) \\
    & = \epsilon(m)+\left\lfloor \frac{2}{3}(m+1)\right\rfloor -s\left(\left\lfloor \frac{2}{3}(m+1)\right\rfloor\right).
\end{align*}

\noindent Case (3)-(a): Write $m+1=3p+2$ for some positive integer $p$.  Since $m \equiv 2 \mod 8$, it follows that $3p+1 \equiv 2 \mod 8$ and so $p \equiv 3 \mod 8$. Thus, $p$ has binary representation $b_r\ldots b_3011$.  Since $\mathbf{j}_{\max}=(j_{n-1},j_n)=(p,2)$, we have $\alpha(n)=(1-p)/2$.  It follows that
\begin{align*}
    -\nu(b_{2,m}) &= \nu(m)+ \sum_{k=n-1}^{n-1} [(n-k+1)j_k - s(j_k)] - s(j_n) - s(-\alpha(n)-1)+ s(j_n-\alpha(n)-1) \\
    &= 1+2j_{n-1}-s(j_{n-1})-s(j_n)-s((p-1)/2-1) + s(j_n+(p-1)/2-1) \\
    & = 1+2p-s(p) -s(2)-s((p-3)/2)+s((p+1)/2) \\
    & = 2p-s(p)-s(p-3)+s(p+1) \\
    & =  2p-s(p)+s(4)\\
    & = 1+(2p+1)-s(2p+1) \\
    & = 1+\lfloor 2p+4/3\rfloor - s(\lfloor 2p+4/3\rfloor) \\
    & = \epsilon(m)+\left\lfloor \frac{2}{3}(m+1)\right\rfloor -s\left(\left\lfloor \frac{2}{3}(m+1)\right\rfloor\right).
\end{align*}

\noindent Case (3)-(b): Write $m+1=3p+2$ for some positive integer $p$.  Since $m \equiv 46 \mod 48$, it follows that $3p+1 \equiv 46 \mod 48$ and so $p \equiv 15 \mod 48$. Thus, $p$ has binary representation $b_r\ldots b_51111$.  Since $\mathbf{j}_{\max}=(j_{n-1},j_n)=(1,p-2,1)$, we have $\alpha(n)=(1-p)/2$. It follows that
\begin{align*}
    -\nu(b_{2,m}) &= \nu(m)+ \sum_{k=n-1}^{n-1} [(n-k+1)j_k - s(j_k)] - s(j_n)-s(-\alpha(n)-1)+s(j_n-\alpha(n)-1) \\
    &= 1+3j_{n-2}-s(j_{n-2})+ 2j_{n-1}-s(j_{n-1})-s(j_n)-s((p-1)/2-1) + s(j_n+(p-1)/2-1) \\
    & = 1+3\cdot 1 -s(1)+2(p-2)-s(p-2)-s(1)-s((p-3)/2)+s((p-1)/2) \\
    & = 2p -2 -s(p-2)-s(p-3)+s(p-1) \\
    & = 2p-2-(s(p)-s(2))-(s(p)-s(3))+(s(p)-s(1)) \\
    & = 2p -s(p) \\
    & = 2p+1-s(2p+1) \\
    & = 0+\lfloor 2p+4/3\rfloor - s(\lfloor 2p+4/3\rfloor) \\
    & = \epsilon(m)+\left\lfloor \frac{2}{3}(m+1)\right\rfloor -s\left(\left\lfloor \frac{2}{3}(m+1)\right\rfloor\right).
\end{align*}
Here, $\epsilon(m)=0$ since $m=2m_0$, where $m_0\equiv -1 \mod 12$.
\\
\

\noindent Case (3)-(c): Write $m+1=3p+2$ for some positive integer $p$.  Since $m \equiv 22 \mod 48$, it follows that $3p+1 \equiv 22 \mod 48$ and so $p \equiv 7 \mod 48$. In this case $\mathbf{j}_{\max}=(j_{n-2},j_{n-1},j_n)=(1,p-2,1)$ and thus the same argument applies as in Case (3)-(b).  This completes the proof of Zagier's conjecture.

\end{proof}

\end{section}

\end{document}